\documentclass{article}
\usepackage[utf8]{inputenc}
\usepackage{amsmath}
\usepackage{amsfonts} 
\usepackage{amsthm}
\usepackage{verbatim}
\usepackage{etoolbox}
\usepackage{dirtytalk}
\usepackage{tikz-cd}
\usepackage{derivative}
\usepackage{amssymb}
\usepackage{mathtools}
\usepackage{authblk}
\usepackage{biblatex}
\theoremstyle{plain}
\newtheorem{thm}{Theorem}[section]
\newtheorem{lem}[thm]{Lemma}
\newtheorem{obs}[thm]{Observation}
\newtheorem{prop}[thm]{Proposition}
\newtheorem{cor}[thm]{Corollary}

\newtheorem*{claim*}{Claim}

\theoremstyle{definition}
\newtheorem{defi}[thm]{Definition}
\newtheorem{ex}[thm]{Example}

\theoremstyle{remark}
\newtheorem{rem}{Remark}[thm]
\newtheorem{step}{Step}

\AtBeginEnvironment{thm}{\setcounter{equation}{0}}
\AtBeginEnvironment{lem}{\setcounter{equation}{0}}
\AtBeginEnvironment{obs}{\setcounter{equation}{0}}
\AtBeginEnvironment{defi}{\setcounter{equation}{0}}
\AtBeginEnvironment{prop}{\setcounter{equation}{0}}
\AtBeginEnvironment{cor}{\setcounter{equation}{0}}

\AtBeginEnvironment{thm}{\setcounter{case}{0}}
\AtBeginEnvironment{lem}{\setcounter{case}{0}}
\AtBeginEnvironment{obs}{\setcounter{case}{0}}
\AtBeginEnvironment{defi}{\setcounter{case}{0}}
\AtBeginEnvironment{prop}{\setcounter{case}{0}}
\AtBeginEnvironment{cor}{\setcounter{case}{0}}

\AtBeginEnvironment{thm}{\setcounter{claim}{0}}
\AtBeginEnvironment{lem}{\setcounter{claim}{0}}
\AtBeginEnvironment{obs}{\setcounter{claim}{0}}
\AtBeginEnvironment{defi}{\setcounter{claim}{0}}
\AtBeginEnvironment{prop}{\setcounter{claim}{0}}
\AtBeginEnvironment{cor}{\setcounter{claim}{0}}

\newcommand{\R}{\mathbb{R}}
\newcommand{\B}{\mathbb{B}}
\newcommand{\s}{\mathbb{S}}
\newcommand{\RR}{\mathcal{R}}
\newcommand{\CC}{\mathcal{C}}

\newcommand{\I}{\mathcal{I}}
\newcommand{\Z}{\mathcal{Z}}

\newcommand{\codim}{\textnormal{codim}}

\newcommand{\id}{\textnormal{id}}
\newcommand{\im}{\textnormal{im}}

\title{Relative Stone-Weierstrass theorem for mappings between varieties}
\author{Juliusz Banecki}
\affil{Faculty of Mathematics and Computer Science,
Jagiellonian University,
ul. Lojasiewicza 6, 30-348 Krakow, Poland\\
\texttt{juliusz.banecki@student.uj.edu.pl} \\
 https://orcid.org/0009-0004-5604-4157}
\date{}

\begin{document}
\maketitle

\renewcommand{\thefootnote}{\fnsymbol{footnote}} 
\footnotetext{\emph{Math subject classification[2020]:} 	14P99}     
\renewcommand{\thefootnote}{\arabic{footnote}} 

\begin{abstract}
We introduce a class of real algebraic varieties characterised by a simple rationality condition, which exhibit strong properties regarding approximation of continuous and smooth mappings by regular ones. They form a natural counterpart to the classes of malleable and quasi-malleable varieties. The approximation property studied here is stronger than those considered before and allows us to deduce non-trivial facts about extensions of regular mappings between varieties.
\end{abstract}

\section{Introduction}\label{sec:intro}
In this paper we deal with \emph{real affine varieties}, defined as in \cite{RAG} or \cite{Mangolte}. Note that in this sense the real projective space $\mathbb{P}^n(\R)$ is an affine variety. Morphisms of such varieties are called \emph{regular maps}. The ring of real valued regular functions of an affine variety $X$ is denoted by $\RR(X)$, while the field of rational functions of an irreducible affine variety is denoted by $\R(X)$.

Throughout the paper we will often deal with mappings defined on arbitrary subsets of real affine varieties, and for this reason we expand the definition of a regular mapping in the following way:

Let $S$ be an arbitrary subset of a real affine variety $X$, and let $Y$ be another affine variety. We say that a mapping $\varphi:S\rightarrow Y$ is regular, if there exists a set $S_0\subset X$ locally closed in Zariski topology, containing $S$ as a subset, and a regular mapping (in the usual sense) $\varphi_0:S_0\rightarrow Y$ such that $\varphi_0\vert_S=\varphi$. 

Recently, real algebraic geometry has seen significant progress in dealing with problems regarding approximation of continuous or smooth mappings between varieties by regular mappings. The most important tools in this area are the classes of \emph{malleable} and \emph{quasi-malleable} varieties, introduced in \cite[Definition 2.1]{real_homogeneous} and \cite[Definition 1.2]{Banecki+Kucharz} respectively. These classes are defined in a rather technical way, so we will not repeat those definitions here. Nonetheless we would like to note the following characterisation of malleable varieties, which highlights their importance in the theory:

\begin{thm}[{\cite[Theorem 1.6]{Banecki+Kucharz}}]\label{mal_char}
Let $Y$ be a malleable variety. Then the following approximation property holds true:

Let $X$ be a non-singular affine variety, $Z\subset X$ a non-singular Zariski closed subvariety of $X$ and $f:X\rightarrow Y$ a smooth mapping. Assume that $f$ is homotopic to a regular mapping $g:X\rightarrow Y$ through a homotopy relative to the set $Z$. Then $f$ can be approximated in the (weak) $\mathcal{C^\infty}$ topology by regular mappings $\widetilde{f}:X\rightarrow Y$, agreeing with $f$ on $Z$.

Moreover this property characterises malleable varieties among all non-singular affine varieties.
\end{thm}
For the definition of the $\CC^\infty$ topology see \cite[p. 36]{differential_top}.

There are other versions of Theorem \ref{mal_char}, and a similar characterisation in terms of an approximation property exists in the case of quasi-malleable varieties; see \cite{Banecki+Kucharz} for more details. The current paper aims to develop this line of research even further, introducing a new class of \emph{uniformly retract rational} varieties, which admit approximation properties stronger than malleable varieties, and which at the same time turn out to be characterised by a simple local rationality-type condition. To introduce our results we first need to define this class of varieties:

\begin{defi}
Let $Y$ be a real affine variety.

We say that $Y$ is \emph{globally retract rational} if there exist a Zariski open subset $W\subset\R^n$ for some $n$ and two regular mappings
\begin{align*}
    i:Y\rightarrow W \\
    r:W\rightarrow Y
\end{align*}
such that $r\circ i=\id_Y$.

We say that $Y$ is \emph{retract rational} if there exists a Zariski open and dense set $V\subset Y$ which is globally retract rational.

Finally we say that $Y$ is \emph{uniformly retract rational} if for every point $y\in Y$ there exists a Zariski open neighbourhood $V$ of $y$ in $Y$ which is globally retract rational.
\end{defi}

While retract rational varieties have been studied since the 1980s, the concept of uniformly retract rational ones, even though natural, seems to be a new one. 

Our main result is the following equivalence between uniform retract rationality and a certain approximation-interpolation property:

\begin{thm}\label{cont_app_thm}
Let $Y$ be a uniformly retract rational affine variety. Then the following approximation property holds true:

Let $X$ be an affine variety, $Z\subset X$ be a Zariski closed subvariety of $X$, $C\subset X$ be a closed and semialgebraic set such that $Z\cap C$ is compact, and $f:C\rightarrow Y$ a continuous mapping satisfying the following two conditions:
\begin{enumerate}
    \item $f$ is homotopic to a regular mapping $g:C\rightarrow Y$,
    \item $f\vert_{Z\cap C}$ is regular. 
\end{enumerate}
Then $f$ can be approximated by regular mappings $\widetilde{f}:C\rightarrow Y$ in the compact-open topology, agreeing with $f$ on $Z\cap C$.

Moreover, this property characterises uniformly retract rational varieties among all affine varieties.
\end{thm}

\begin{rem}
The assumption of $C$ being closed and semialgebraic in Theorem \ref{cont_app_thm} is essential. Nonetheless, it is not clear whether the assumption on $Z\cap C$ being compact can be dropped; see Section \ref{sec:examples} for the relevant counterexamples.
\end{rem}

There is also a version of Theorem \ref{cont_app_thm} taking into account approximation in the $\mathcal{C}^\infty$ topology. To state it, we first need to take note of the following observation, the proof of which is postponed until the next section:
\begin{obs}\label{urr=>nonsingular}
Uniformly retract rational varieties are non-singular.
\end{obs}
The second result of the paper can now be formulated as the following:

\begin{thm}\label{smooth_app_thm}
Let $Y$ be a uniformly retract rational affine variety. Then the following approximation property holds true:

Let $X$ be a non-singular affine variety, $Z\subset X$ be a Zariski closed non-singular subvariety of $X$, $C\subset X$ be a closed and semialgebraic set such that $Z\cap C$ is compact, $D\subset C$ a set open in $X$, and $f:C\rightarrow Y$ a continuous mapping satisfying the following three conditions:
\begin{enumerate}
    \item $f$ is homotopic to a regular mapping $g:C\rightarrow Y$,
    \item $f\vert_{Z\cap C}$ is regular,
    \item $f\vert_D$ is smooth.
\end{enumerate}
Then $f$ can be approximated by regular mappings $\widetilde{f}:C\rightarrow Y$ in the compact-open topology, which simultaneously approximate $f\vert_D$ in the $\CC^\infty$ topology and agree with $f$ on $Z\cap C$.
\end{thm}
\begin{rem}\label{faith}
Theorem \ref{smooth_app_thm} can be strengthened further, at the cost of it becoming more technical. Recall from \cite[Section 4]{Banecki+Kucharz}, that a Zariski closed subvariety $Z$ of a non-singular variety $X$ is said to be \emph{faithful} at a point $x\in Z$ if 
\begin{equation*}
    \I_{\mathcal{O}_x}(Z_x)=\I_{\RR(X)}(Z)\mathcal{O}_x
\end{equation*}
where $\mathcal{O}_x$ is the ring of all real analytic function-germs $(X,x)\rightarrow \R$ and $\I_{\mathcal{O}_x}(Z_x)$ is the ideal in $\mathcal{O}_x$ comprised of all function-germs vanishing on the germ $Z_x$ of $Z$ at $x$, while $\I_{\RR(X)}(Z)\mathcal{O}_x$ is the extension of the ideal $\I_{\RR(X)}(Z)$ in $\RR(X)$ via the canonical ring homomorphism $\RR(X)\rightarrow \mathcal{O}_x$. With this definition in mind, Theorem \ref{smooth_app_thm} holds without the assumption on $Z$ being non-singular, if instead it is assumed faithful at every point of $Z\cap D$.
\end{rem}

To illustrate the strength of our results we present a few simple applications. Most notably, they imply non-trivial facts regarding extensions of regular mappings from subvarieties:
\begin{cor}
Let $Y$ be a uniformly retract rational affine variety and let $X$ be an affine variety. Let $Z\subset X$ be a compact Zariski closed subvariety of $X$ and $f,g:X\rightarrow Y$ be two homotopic regular mappings. Suppose that $g$ extends to a regular mapping $G:X\rightarrow Y$. Then $f$ extends to a regular mapping $F:X\rightarrow Y$ as well.
\begin{proof}
Let $H:Z\times[0,1]\rightarrow Y$ be a homotopy between $g$ and $f$. By Borsuk homotopy extension theorem (\cite[p. 94]{Borsuk}) we get that $H$ can be extended to a homotopy $\widehat{H}:X\times [0,1]\rightarrow Y$ such that $\widehat{H}(x,0)=G(x)$ for $x\in X$ and $\widehat{H}(x,1)=f(x)$ for $x\in Z$. We can now apply Theorem \ref{cont_app_thm} to obtain that $x\mapsto \widehat{H}(x,1)$ can be approximated by regular mappings agreeing with $f$ on $Z$. Any such approximation is the desired extension of $f$.
\end{proof}
\end{cor}
Another interesting corollary of our results is the following:
\begin{cor}
Let $Y$ be a uniformly retract rational affine variety and let $X$ be a compact affine variety. Let $f,g:X\rightarrow Y$ be two homotopic mappings. Then, there exists a regular mapping $F:X\times\R\rightarrow Y$ such that $F(x,0)=f(x),\;F(x,1)=g(x)$ for $x\in X$. In other words the mappings are homotopic through a homotopy which itself is a regular mapping.
\begin{proof}
Let $H:X\times [0,1]\rightarrow Y$ be a homotopy between $f$ and $g$. We extend $H$ to a mapping $\widehat{H}:X\times\R\rightarrow Y$ defined by the formula:
\begin{equation*}
    \widehat{H}(x,t):=
    \begin{cases}
        H(x,0) & t<0 \\
        H(x,t) & 0\leq t \leq 1 \\
        H(x,1) & t>1
    \end{cases}
\end{equation*}
The mapping $\widehat{H}$ is homotopic to the regular mapping $(x,t)\mapsto f(x)$ through the homotopy
\begin{equation*}
    (x,t,\theta)\mapsto \widehat{H}(x,t\theta).
\end{equation*}
Finally, it suffices to apply Theorem \ref{cont_app_thm} to $\widehat{H}$ with $C:=X\times \R$ and $Z:=X\times\{0\}\cup X\times \{1\}$.
\end{proof}
\end{cor}

To end this section we would like to summarise the known implications between the classes of uniformly retract rational, malleable and quasi-malleable varieties. To make the picture complete we would like to recall four other classes of varieties related to the topic:
\begin{defi}
Let $Y$ be an affine variety.

We say that $Y$ is \emph{rational} if there exists a Zariski open and dense subset $V$ of $Y$ which is biregularly isomorphic to a Zariski open subset $W$ of $\R^n$ for some $n$.

We say that $Y$ is \emph{uniformly rational} if every point $y\in Y$ admits a Zariski open neighbourhood $V$ in $Y$ which is biregularly isomorphic to a Zariski open subset $W$ of $\R^n$ for some $n$.

We say that $Y$ is \emph{dominable at every point} if for every $y\in Y$ there exists a natural number $n$ and a regular mapping $\varphi:\R^n\rightarrow Y$ such that $\varphi(0)=y$ and the derivative $d_0\varphi:T_0\R^n\rightarrow T_yY$ is surjective.

Finally, we say that $Y$ is \emph{unirational} if there exists a Zariski open subset $W$ of $\R^n$ for some $n$ and a regular mapping $\varphi:W\rightarrow Y$ such that the image of $\varphi$ is Zariski dense in $Y$.
\end{defi}

Let $Y$ be a non-singular irreducible affine variety. The following diagram illustrates the known implications between the different properties regarding $Y$:
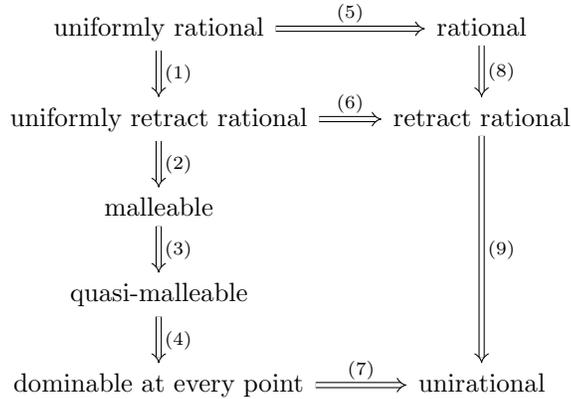
\begin{center}
\begin{tikzcd}
\text{uniformly rational} \arrow[r, "(5)", Rightarrow] \arrow[d, "(1)", Rightarrow]         & \text{rational} \arrow[d, "(8)", Rightarrow]           \\
\text{uniformly retract rational} \arrow[d, "(2)", Rightarrow] \arrow[r, "(6)", Rightarrow] & \text{retract rational} \arrow[ddd, "(9)", Rightarrow] \\
\text{malleable} \arrow[d, "(3)", Rightarrow]                                               &                                                        \\
\text{quasi-malleable} \arrow[d, "(4)", Rightarrow]                                         &                                                        \\
\text{dominable at every point} \arrow[r, "(7)", Rightarrow]                                       & \text{unirational}                                    
\end{tikzcd}
\end{center}

The only non-trivial implication in the diagram is $(2)$, which follows from Theorems \ref{urr_equivalences} and \ref{smalleable=>malleable}.

In Section \ref{sec:examples} we provide an explicit example of a malleable variety, which is not retract rational. This shows that the arrows $(2)$ and $(9)$ cannot be reversed.

There also does exist an example of a uniformly retract rational variety which is not rational, forcing the arrows $(1)$ and $(8)$ to be strict. It is based on a rather sophisticated construction given in \cite{counterex}. The complex algebraic hypersurface $V$ in $\mathbb{C}^4$ described in \cite[p. 315, Example 2]{counterex} is not rational. Since $V$ is defined over $\R$, the real algebraic hypersurface $V(\R) := V \cap \R^4$ in $\R^4$ is not rational. On the other hand, by \cite[p. 299, Th\'eor\`eme 1’]{counterex} there exists a nonempty Zariski open subset of $V (\R) \times \R^3$ that is biregularly isomorphic to a Zariski open subset of $\R^6$. Thus, some nonempty Zariski open subset $Y$ of $V (\R)$ is a globally retract rational real affine variety. Clearly, $Y$ is not rational. This result was brought to the author's attention by Wojciech Kucharz, who originally was informed about it by Olivier Benoist.

To the author's best knowledge nothing is known about the remaining implications, and in each case it is not at all clear what to expect. The question whether $(5)$ is strict had been considered before by Gromov in \cite{Gromov} and it has been open ever since.

%--------------------------------------------------------------------

\section{Preliminaries}
We begin by presenting the proof of Observation \ref{urr=>nonsingular}:

\begin{lem}\label{grr_embedding_lem}
Let $Y$ be a globally retract rational variety. Fix an embedding of $Y$ as a closed subvariety of $\R^m$ for some $m$. Then there exists a Zariski open set $W'\subset \R^m$ containing $Y$ together with a regular retraction $r':W'\rightarrow Y$, $r'\vert_Y =\id$.
\begin{proof}
Let $W\subset \R^n$ and $i:Y\rightarrow W$, $r:W\rightarrow Y$ be as in the definition of global retract rationality. The mapping $i$ admits a regular extension $\widetilde{i}:W'\rightarrow W$ to a Zariski open set $W'\subset \R^m$ containing $Y$. It suffices to define $r':=r\circ \widetilde{i}$.
\end{proof}
\end{lem}

\begin{proof}[Proof of Observation \ref{urr=>nonsingular}]
After passing to a neighbourhood it suffices to show that a globally retract rational variety $Y$ is non-singular. Applying Lemma \ref{grr_embedding_lem}, we can embed $Y$ in a Zariski open set $W\subset\R^n$ such that there exists a regular retraction $r:W\rightarrow Y$. 

First of all we claim that $Y$ is irreducible. Suppose that $Y$ splits as a union $Y=Y_1\cup Y_2$ where $Y_1,Y_2$ are Zariski closed subset of $Y$. Then $W$ splits as $W=r^{-1}(Y_1)\cup r^{-1}(Y_2)$, and since it is irreducible we get $W=r^{-1}(Y_1)$ or $W=r^{-1}(Y_2)$. It follows that $Y=Y_1$ or $Y=Y_2$, so $Y$ is indeed irreducible.

\begin{claim*}
Let now $m$ be the dimension of $Y$. Considering $r$ as a mapping from $W$ to $\R^n$, the rank of its derivative $dr$ is at most $m$ at every point of $W$.
\begin{proof}
Suppose the contrary, and let $m_0>m$ be the maximal rank of $dr$. From a version of the rank theorem for Nash mappings it then follows that $Y=\im \;r$ contains a subset which is Nash diffeomorphic to the $m_0$-dimensional hypercube. This contradicts the fact that the semialgebraic dimension of $Y$ is equal to its dimension as an affine variety, i.e. $m$.
\end{proof}
\end{claim*}
Let now $a$ be an arbitrary point of $Y$. Consider the Zariski tangent space $T_a^\text{Zar}(Y)$ to $Y$ at $a$. As $r$ agrees with the identity $\id:W\rightarrow W$ on $Y$, its derivative must agree with the derivative of the identity when restricted to the space $T_a^\text{Zar}(Y)$:
\begin{equation*}
    d_ar\vert_{T_a^\text{Zar}(Y)}=\id\vert_{T_a^\text{Zar}(Y)}.
\end{equation*}
As the rank of the left hand side is at most $m$, it follows that
\begin{equation*}
    \dim(T_a^\text{Zar}(Y))\leq m
\end{equation*}
which implies that $a$ is a regular point of $Y$.
\end{proof}

The following definition is crucial for the rest of the paper.
\begin{defi}
Let $Y$ be an affine variety. A strong dominating spray over $Y$ is a quintuple $(n, M, N, \sigma,\tau)$, where $n$ is a nonnegative integer, $M$ is a Zariski open subset of $Y\times\R^n$ containing $Y\times\{0\}$, $N\subset Y\times Y$ is a Zariski open subset of $Y\times Y$ containing the diagonal $\triangle_Y$ of $Y\times Y$ and $\sigma:M\rightarrow Y,\;\tau:N\rightarrow\R^n$ are regular mappings satisfying:
\begin{enumerate}
    \item $(y,\tau(y,z))\in M$ and $\sigma(y,\tau(y,z))=z$ for all $(y,z)\in N$,
    \item $\tau(y,y)=0$ for $y\in Y$.
\end{enumerate}
Note that the two conditions further imply that
\begin{equation*}
    \sigma(y,0)=y \text{ for every }x\in Y.
\end{equation*}
\end{defi}

\begin{defi}\label{def_of_smalleable}
We say that a non-singular real affine variety $Y$ is \emph{strongly malleable} if there exists a strong dominating spray over $Y$.

We say that $Y$ is \emph{locally strongly malleable} if every point $a\in Y$ admits a Zariski open neighbourhood $V$ in $Y$ which is strongly malleable. 
\end{defi}

The following examples illustrate two cases in which a strong dominating spray over a variety can be given explicitly. Later in the paper we will present a way less explicit, yet more general method of creating such sprays.
\begin{ex}
Let $G$ be a rational real algebraic group, i.e. a real affine variety, birationally equivalent to $\R^n$ for some $n$, endowed with a structure of a group such that the group operations $G\times G\rightarrow G$, $(a,b)\rightarrow a\cdot b$, and $G\rightarrow G$, $a\rightarrow a^{-1}$ are regular mappings. Then $G$ is strongly malleable.
\begin{proof}
As $G$ is rational, there exists a Zariski open set $V\subset G$ biregularly isomorphic to a Zariski open set $W\subset \R^n$ through an isomorphism $\varphi:V\rightarrow W$. After translations we may assume that $e\in V$, $0\in W$ and $\varphi(e)=0$, where $e$ is the neutral element of $G$. A strong dominating spray over $G$ can now be defined as follows:
\begin{align*}
    M&:=G\times W, \\
    N&:= \{(a,b)\in G\times Y: a^{-1}\cdot b \in V\},\\
    \tau(a,b)&:=\varphi(a^{-1}\cdot b), \\
    \sigma(a,v)&:=a\cdot \varphi^{-1}(v).
\end{align*}
\end{proof}
\end{ex}

\begin{ex}
Unit spheres of all dimensions are strongly malleable.
\begin{proof}
We can define a strong dominating spray over the $n$-dimensional unit sphere in the following way:
\begin{align*}
    &M:=\{(y,v)\in \s^n\times\R^{n+1}:v\neq -2y\},\\
    &N:=\{(y,z)\in \s^n\times\s^n: z\neq -y\},\\
    &\tau(y,z):=z-y, \\
    &\begin{multlined}
         \sigma(y,v):=-y+\frac{<2y,v+2y>}{||v+2y||^2}(v+2y)=\\
    =\text{ the intersection of the line passing through }-y \text{ and } y+v \\
    \text{ and the sphere (other than $-y$)}.
    \end{multlined}
\end{align*}
\end{proof}
\end{ex}

We will deduce Theorem \ref{cont_app_thm} directly from the following more general characterisation of uniformly retract rational varieties:

\begin{thm}\label{urr_equivalences}
Let $Y$ be an affine variety. The following conditions are equivalent:
\begin{enumerate}
    \item $Y$ satisfies the approximation property of Theorem \ref{cont_app_thm}, \label{(ap)}
    \item $Y$ is uniformly retract rational, \label{(urr)}
    \item $Y$ is locally strongly malleable, \label{(lsmalleable)}
    \item $Y$ is strongly malleable. \label{(smalleable)}
\end{enumerate}
\end{thm}
In the remainder of this section we will prove the equivalences \begin{equation*}
    (\ref{(ap)})\implies (\ref{(urr)})\iff (\ref{(lsmalleable)}).
\end{equation*} 
Then in Section \ref{sec:gluing} we prove $(\ref{(lsmalleable)})\implies (\ref{(smalleable)})$ and in Section \ref{sec:C0} we prove $(\ref{(smalleable)})\implies (\ref{(ap)})$.

%-------------------------------------------------------------------------------------

\begin{obs}\label{property=>urr}
If an affine variety $Y$ satisfies the approximation property of Theorem \ref{cont_app_thm}, then it is uniformly retract rational. 
\begin{proof}
Choose a point $a\in Y$. Fix an embedding $Y\subset\R^n$. By \cite[Theorem 1]{retraction}, there exists an open Euclidean neighbourhood $U$ of $Y$ in $\R^n$ which retracts onto $Y$ through a continuous semialgebraic retraction $p:U\rightarrow Y$. Fix a closed ball $\overline{\B}(a,\varepsilon)$ centered at $a$ contained in $U$. The ball is contractible, so $p\vert_{\overline{\B}(a,\varepsilon)}$ is homotopically trivial. We can apply the assumption on $Y$ (with $X=\R^n, C=\overline{\B}(a,\varepsilon), Z=Y$), to find a regular approximation $\widetilde{p}:\overline{\B}(a,\varepsilon)\rightarrow Y$ of $p\vert_{\overline{\B}(a,\varepsilon)}$, equal to the identity on $\overline{\B}(a,\varepsilon)\cap Y$. The mapping $\widetilde{p}$ extends to a regular mapping $\widetilde{r}:\widetilde{U}\rightarrow Y$, where $\widetilde{U}$ is a Zariski open subset of $\R^n$. Finally taking $V:=\widetilde{U}\cap Y$ we get that $V$ is a globally retract rational variety containing $a$, with the retraction defined by 
\begin{align*}
    W&:=\widetilde{r}^{-1}(V), \\
    r&:=\widetilde{r}\vert_{W}, \\
    i&: V\rightarrow W \text{ being the inclusion.}
\end{align*}
\end{proof}
\end{obs}

\begin{obs}\label{uniform_retract=>locally_smalleable}
Uniformly retract rational varieties are locally strongly malleable.
\begin{proof}
The question is local so it suffices to prove that a globally retract rational variety $Y$ is strongly malleable. Take $W\subset \R^n$, $i:Y\rightarrow W$, $r:W\rightarrow Y$ as in the definition of global retract rationality. It now suffices to define:
\begin{align*}
    N&:=Y\times  Y \\
    M&:= \{(y,v)\in Y\times \R^n: i(y)+v\in W\} \\
    \sigma(y,v)&:= r(i(y)+v) \\
    \tau(y,z)&:= i(z)-i(y).
\end{align*}
\end{proof}
\end{obs}

\begin{obs}\label{lsmalleable=>urr}
Locally strongly malleable varieties are uniformly retract rational.
\begin{proof}
Again after localising it suffices to prove that a strongly malleable variety $Y$ is uniformly retract rational. Let $(n,M,N,\sigma,\tau)$ be a strong dominating spray over $Y$. Fix a point $y\in Y$ and define:
\begin{align*}
    V&:=\{z\in Y: (y,z)\in N\},\\
    W&:=\{v\in \R^n:(y,v)\in M\},\\
    i(z)&:=\tau(y,z),\\
    r(v)&:=\sigma(y,v).
\end{align*}
\end{proof}
\end{obs}

Lastly, we note that strongly malleable varieties are malleable in the sense of \cite[Definition 2.1]{real_homogeneous}. To prove that, it is convenient to introduce the following lemma, which will be also useful later on.
\begin{lem}\label{ball_in_M}
Let $Y$ be a strongly malleable variety. Then there exists a strong dominating spray $(n,M',N',\sigma',\tau')$ over $Y$ which additionally satisfies 
\begin{equation*}
    Y\times \B(0,1)\subset M'.
\end{equation*}
\begin{proof}
Let $(n,M,N,\sigma,\tau)$ be some strong dominating spray over $Y$ and assume that it does not satisfy the property $Y\times \B(0,1)\subset M$. Consider the function
\begin{align*}
    \varphi&:Y\rightarrow \R^+ \\
    \varphi(y)&:=\textnormal{dist}((y,0),\partial M)
\end{align*}
where $\partial M$ denotes the boundary $\overline{M}\backslash M$ of $M$. The function is semialgebraic and continuous, so from Łojasiewicz's inequality (\cite[Proposition 2.6.2]{RAG}) we deduce that there exists a regular function $f\in\RR(Y)$ which satisfies
\begin{equation*}
    0<f(y)<\varphi(y) \textnormal{ for every }y\in Y.
\end{equation*}
It now suffices to define a new strong dominating spray over $Y$ in the following way:
\begin{align*}
    M'&:=\{(y,v)\in Y\times\R^n:(y,f(y)v)\in M\}, \\
    N'&:=N,\\
    \sigma'(y,v)&:=\sigma(y,f(y)v), \\
    \tau'(y,z)&:=\frac{\tau(y,z)}{f(y)}.
\end{align*}
\end{proof}
\end{lem}

\begin{prop}\label{smalleable=>malleable}
Strongly malleable varieties are malleable.
\begin{proof}
Note that thanks to Observations \ref{lsmalleable=>urr} and \ref{urr=>nonsingular} a strongly malleable variety $Y$ is non-singular. Let $(n,M,N,\sigma,\tau)$ be a strong dominating spray over $Y$, satisfying the condition $Y\times\mathbb{B}(0,1)\subset M$ of Lemma \ref{ball_in_M}. It is now easy to realise that the triple $(Y\times \R^n,p,s)$ is a dominating spray over $Y$, where $p:Y\times\R^n\rightarrow Y$ is the projection onto the first coordinate and $s:Y\times\R^n\rightarrow Y$ is defined by
\begin{equation*}
    s(y,v):=\sigma\left(y,\frac{v}{1+||v||^2}\right).
\end{equation*}
\end{proof}
\end{prop}

\section{Gluing strong dominating sprays}\label{sec:gluing}
In this section we prove the implication $(\ref{(lsmalleable)})\implies (\ref{(smalleable)})$ of Theorem \ref{urr_equivalences}.
\begin{obs}\label{gluing_obs1}
Let $Y$ be an irreducible variety, $n$ a natural number and $U\subset Y\times\R^n$ a Zariski open set. Then, the ideal $I:=\I(U\cap (Y\times \{0\}))$ of $\RR(U)$ is generated by $v_1,\dots,v_n$, where $v_1,\dots,v_n$ are the coordinates of $\R^n$.
\begin{proof}
First, since $\RR(U)$ is a localisation of $\RR(Y\times \R^n)$ it suffices to verify the claim when $U=Y\times\R^n$. Fixing an embedding $Y\subset\R^m$ it suffices to prove the claim only for elements of $I\cap \mathcal{P}(Y)[v_1,\dots,v_n]$, where $\mathcal{P}(Y)$ is the polynomial ring of $Y$. After such reduction the claim follows easily.
\end{proof}
\end{obs}
The following lemma is fundamental, it has appeared before in a slightly different version as \cite[Lemma 6.4]{Banecki+Kucharz}. It will find its use twice in the proof of the implication $(\ref{(lsmalleable)})\implies(\ref{(smalleable)})$ of Theorem \ref{urr_equivalences}.
\begin{lem}\label{gluing_lem1}
Let $Y$ be an irreducible variety and let $n$ be a natural number. Let $U\subset Y\times\R^n$ be a Zariski open set such that $U\cap (Y\times\{0\})\neq\emptyset$. Let $f\in\RR(U)$ and let $P,Q\in\R(Y\times\R^n)$ be two rational functions, regular on a Zariski neighbourhood of $Y\times\{0\}$ satisfying
\begin{equation*}
    Qf=P\textnormal{ on a Zariski neighbourhood of } U\cap (Y\times\{0\}).
\end{equation*}
Suppose further that $f\vert_{U\cap (Y\times\{0\})}$ admits a regular extension onto the entire variety $Y\times\{0\}$. Then, the function
\begin{equation*}
    g(y,v):=f(y,Q(y,0)v)
\end{equation*}
defined on a Zariski neighbourhood of $U\cap (Y\times\{0\})$ admits a regular extension to a Zariski neighbourhood of $Y\times\{0\}$.
\begin{proof}
Using the assumption on $f$, we deduce that there exists $F\in\RR(Y\times\R^n)$ satisfying $f\vert_{U\cap (Y\times\{0\})}=F\vert_{U\cap (Y\times\{0\})}$. After substituting $f':=f-F$ and $P'=P-QF$ we can assume that $f$ vanishes on $U\cap (Y\times\{0\})$. Hence $P$ vanishes on $Y\times\{0\}$, so by Observation \ref{gluing_obs1} it can be written as $\sum_i v_iP_i$
for some $P_i$ regular on a Zariski neighbourhood of $Y\times\{0\}$. Similarly, we have that $Q(y,v)-Q(y,0)$ vanishes on $Y\times\{0\}$ so it can be written as $\sum_i v_iQ_i$ for some $Q_i$ regular on a Zariski neighbourhood of $Y\times\{0\}$. Altogether we have
\begin{multline*}
    f(y,Q(y,0)v)= \frac{P(y,Q(y,0)v)}{Q(y,Q(y,0)v)}=\\
    =\frac{\sum_i Q(y,0)v_iP_i(y,Q(y,0)v)}{Q(y,0)+\sum_i Q(y,0)v_iQ_i(y,Q(y,0)v)} 
    = \frac{\sum_iv_iP_i(y,Q(y,0)v)}{1+\sum_i v_iQ_i(y,Q(y,0)v)}.
\end{multline*}
The above denominator does not vanish at any point of $Y\times\{0\}$, hence the fraction gives a regular representation of $g$ on a Zariski neighbourhood of $Y\times\{0\}$.
\end{proof}
\end{lem}

\begin{thm}\label{lsmalleable=>smalleable}
Locally strongly malleable varieties are strongly malleable.
\begin{proof}
Let $Y$ be a locally strongly malleable variety. Thanks to Observations \ref{lsmalleable=>urr} and \ref{urr=>nonsingular}, $Y$ is non-singular, so it is a disjoint union of its irreducible components. It is easy to see that it suffices to prove that each of the components is strongly malleable, so we assume that $Y$ is irreducible. By induction we can also assume that $Y=Y_1\cup Y_2$ with $Y_1,Y_2$ being Zariski open subsets of $Y$ which are strongly malleable. Let $(n_i, M_i,N_i,\sigma_i,\tau_i)$ for $i=1,2$ be strong dominating sprays over $Y_1,Y_2$ respectively. We naturally treat the sets $N_1, N_2$ as Zariski open subsets of $Y\times Y$, and similarly we treat $M_1,M_2$ as subsets of $Y\times \R^{n_1},Y\times \R^{n_2}$ respectively. We will construct a strong dominating spray over $Y$ in a few steps.

\begin{step}
Assuming that $Y\times\{0\}\subset M_1,Y\times\{0\}\subset M_2$.
\end{step}
So far $M_1$ is by definition a subset of $Y_1\times\R^{n_1}\subset Y\times \R^{n_1}$. We will alter the spray over $Y_1$ slightly, so that the entire set $Y\times \{0\}$ is contained in $M_1$ and the conditions
\begin{gather*}
    \triangle_{Y_1}\subset N_1, \\
    (y,\tau_1(y,z))\in M_1\text{ and }\sigma_1(y,\tau_1(y,z))=z\text{ for all }(y,z)\in N_1, \\
    \tau_1(y,y)=0\text{ for }y\in Y_1    
\end{gather*}
are still satisfied.

To do that, fix an embedding of $Y$ as a Zariski closed subset of $\R^m$ for some $m$ and take $Q\in\RR(Y\times \R^{n_1})$, $\Z(Q)\cap M_1=\emptyset$ to be a common denominator of all coordinates of $\sigma_1$, i.e. a regular function such that the product
\begin{equation*}
    P:=Q\sigma_1
\end{equation*}
extends to a regular mapping $P:Y\times\R^{n_1}\rightarrow \R^m$. Define the new spray over $Y_1$ by
\begin{align*}
    N_1'&:=N_1,\\
    \sigma_1'(y,v)&:=\sigma_1(y,Q(y,0)v),\\ 
    \tau_1'(y,z)&:=\frac{\tau_1(y,z)}{Q(y,0)}.
\end{align*}
We now also have to redefine $M_1$. We can define it as the maximal domain consisting of points at which the mapping $\sigma_1'$ is regular, which thanks to Lemma \ref{gluing_lem1} does contain $Y\times\{0\}$.

We also perform an analogous reduction in the case of $M_2$.

\begin{step}\label{step_def_spray}
Defining the spray over $Y$.
\end{step}
We will now gradually define all the components of a strong dominating spray over $Y$. First of all take $n:=n_1+n_2$. Now define $M$ and $\sigma$ by the following formulas:
\begin{align*}
    M&:=\{(y,v,w)\in Y\times \R^{n_1}\times\R^{n_2}:(y,v)\in M_1,(\sigma_1(y,v),w)\in M_2\},\\
    \sigma(y,v,w)&:=\sigma_2(\sigma_1(y,v),w).
\end{align*}
It is clear that $M$ is a Zariski open neighbourhood of $Y\times \{0\}$ in $Y\times\R^n$.

Let now $\widetilde{N}$ be some Zariski neighbourhood of the diagonal $\triangle_Y$ in $Y\times Y$ and $\beta:\widetilde{N} \rightarrow \R^{n_1}$ be some regular mapping which vanishes on the diagonal, both to be defined later.

We define $\tau$ by the following formula:
\begin{equation}\label{def_of_tau}
    \tau(y,z):=\bigl(\beta(y,z),\tau_2(\sigma_1(y,\beta(y,z)),z)\bigr).
\end{equation}
This definition makes sense only on some Zariski neighbourhood of $\triangle_{Y_2}$; on this neighbourhood we have
\begin{equation*}
    \sigma(y,\tau(y,z))=\sigma_2\bigl(\sigma_1(y,\beta(y,z)),\tau_2(\sigma_1(y,\beta(y,z)),z)\bigr)=z.
\end{equation*}
If now we can define $\beta$ in such a way that the mapping $\tau$ extends as a regular mapping to a Zariski neighbourhood $N$ of the entire diagonal $\triangle_Y$, then $(n,M,N,\sigma,\tau)$ is a strong dominating spray over $Y$ and we are done.

\begin{step}
Defining $\beta$.
\end{step}
Choose $Q_1\in\RR(Y\times Y)$ to be a denominator of $\tau_1$, i.e. a regular function satisfying $\Z(Q_1)\cap N_1=\emptyset$ such that the product
\begin{equation*}
    P_1:=Q_1\tau_1
\end{equation*}
extends to a regular mapping $P_1:Y\times Y\rightarrow \R^{n_1}$. Choose $Q_2\in\RR(Y\times Y)$ in a similar way so that $\Z(Q_1)\cap N_2=\emptyset$ and $P_2:=Q_2\tau_2$ extends to a regular mapping $P_2:Y\times Y\rightarrow \R^{n_2}$.

Now note that the functions
\begin{align*}
    (y,z)&\mapsto Q_1(y,z),\\
    (y,z)&\mapsto Q_2(z,z)
\end{align*}
do not share a common zero in the set $\widetilde{N}:=N_1\cup (Y\times Y_2)$. This means that we can apply the Chinese reminder theorem to find a function $\theta\in\RR(\widetilde{N})$ satisfying the conditions
\begin{align*}
    Q_1(y,z)&\text{ divides }\theta(y,z) \text{ in }\RR(\widetilde{N}),\\
    Q_2(z,z)&\text{ divides }\theta(y,z)-1 \text{ in }\RR(\widetilde{N}).
\end{align*}
Finally, define $\beta(y,z):=\theta(y,z)\tau_1(y,z)$. As $Q_1$ divides $\theta$, we have that $\beta$ extends as a regular mapping to the entire set $\widetilde{N}$ and vanishes on the diagonal. Finally, it remains to show that under such a definition of $\beta$ the mapping $\tau$ extends to a Zariski neighbourhood of the entire diagonal $\triangle_Y$.
\begin{step}
    Verifying that $\tau$ is defined correctly.
\end{step}
As noted in Step \ref{step_def_spray}, $\tau$ a priori is defined correctly on a neighbourhood of $\triangle_{Y_2}$ so it remains to show that it extends as a regular mapping to a neighbourhood of $\triangle_{Y_1}$.

Consider the following auxiliary mapping
\begin{equation*}
    h(y,z,v):=\tau_2(\sigma_1(y,v+\tau_1(y,z)),z)
\end{equation*}
defined on some Zariski neighbourhood $U$ of $\Delta_{Y_1\cap Y_2}\times\{0\}$ in $N_1\times\R^{n_1}$. On $U$ we have
\begin{equation*}
    Q_2(\sigma_1(y,v+\tau_1(y,z)),z)h(y,z,v)=P_2(\sigma_1(y,v+\tau_1(y,z)),z).
\end{equation*}
Since $h$ is constantly equal to $0$ on $U \cap (N_1\times \{0\})$, we can apply Lemma \ref{gluing_lem1} to all coordinates of $h$ to obtain that
\begin{equation*}
    g(y,z,v):=h(y,z,Q_2(\sigma_1(y,0+\tau_1(y,z)),z)v)=h(y,z,Q_2(z,z)v)
\end{equation*}
extends as a regular mapping to a Zariski neighbourhood of $N_1\times\{0\}$ in $N_1\times\R^{n_1}$. Finally the second component of $\tau$ can be written as
\begin{equation*}
    h(y,z,\beta(y,z)-\tau_1(y,z))=g\left(y,z,\frac{\theta(y,z)-1}{Q_2(z,z)}\tau_1(y,z)\right)
\end{equation*}
which extends as a regular mapping to a Zariski neighbourhood of $\Delta_{Y_1}$.
\end{proof}
\end{thm}

\section{$\mathcal{C}^0$ approximation}\label{sec:C0}
In this section we prove the implication $(\ref{(smalleable)})\implies (\ref{(ap)})$ of Theorem \ref{urr_equivalences}.

\begin{lem}\label{extension_lem}
Let $X$ be an affine variety, $Z\subset X$ be a closed subvariety, and $C\subset X$ be a closed semialgebraic subset of $X$ such that $Z\cap C$ is compact. Let $f:Z\cap C\rightarrow \R^n$ be a regular mapping. Then, there exists a regular mapping $F:C\rightarrow \R^n$, such that $F\vert_{Z\cap C}=f$ and $\sup_{x\in C}||F(x)||\leq 2\sup_{x\in Z\cap C}||f(x)||$.
\begin{proof}
To avoid the trivial case, we assume that $f$ is not constantly equal to zero. First of all note that by definition $f$ extends as a regular mapping to a Zariski locally closed subset $Z_0$ of $Z$ containing $Z\cap C$. From this set it can be extended to a regular mapping $G:C\rightarrow\R^n$, which we will improve further so that it satisfies the desired bound.

For now suppose that $C$ is not compact. Let $\varphi:X\rightarrow \R$ be some regular mapping vanishing precisely on the set $Z$. From Łojasiewicz's inequality (\cite[Proposition 2.6.4]{RAG}) we deduce that there exists an integer $N_0\geq 0$ such that the mapping $\psi:C\rightarrow \R$, $\psi(x):=(1+||x||^2)^{N_0}\varphi^2(x)$ escapes to infinity at infinity. Again by Łojasiewicz's inequality there exists an integer $N_1\geq 0$ such that $H:=(1+\psi)^{-N_1}G$ goes to zero at infinity. The set $A:=\{x\in C: ||H(x)||\geq2\sup_{x\in Z\cap C}||f(x)||\}$ is now compact and disjoint from $Z$. This implies that we can find an exponent $N_2\geq 0$ such that $F:=(1+\psi)^{-N_2}G$ satisfies $\sup_{x\in A}||F(x)||\leq 2\sup_{x\in Z\cap C}||f(x)||$. It follows that $F$ is the desired mapping that we are looking for.

If $C$ is compact the situation is easier, we can just apply the last step of the reasoning (i.e. take $N_0=N_1=0)$. In this case it is not even important that $C$ is semialgebraic.
\end{proof}
\end{lem}

We now are in a position to prove the aforementioned implication:
\begin{prop}\label{cont_app_prop}
Let $Y$ be a strongly malleable variety. Then it satisfies the approximation property of Theorem \ref{cont_app_thm}    
\begin{proof}
Let $X,Z,C,f:C\rightarrow Y$ be as in Theorem \ref{cont_app_thm}, and let $(n,M,N,\sigma,\tau)$ be a strong dominating spray over $Y$. Thanks to Lemma \ref{ball_in_M} we assume that $Y\times\mathbb{B}(0,1)\subset M$. 

As noted in Proposition \ref{smalleable=>malleable}, strongly malleable varieties are malleable. Hence, applying \cite[Theorem 1.14]{Banecki+Kucharz}, $f$ can be approximated by regular mappings in the compact-open topology. Now choose $h:C\rightarrow Y$ to be a sufficiently close regular approximation of $f$ in this topology. We assume that $h$ is so close to $f$ that 
\begin{equation*}
    (h,f)(Z \cap C)\subset N.
\end{equation*}
Define
\begin{equation*}
    \varepsilon:=\sup_{x\in Z\cap C} ||\tau(h(x),f(x))||.
\end{equation*}
Again assuming that $h$ is close enough to $f$ we infer $\varepsilon<1/2$. We can now apply Lemma \ref{extension_lem} to the mapping 
\begin{align*}
    \varphi&:Z\cap C\rightarrow \R^n \\
    \varphi(x)&:=\tau(h(x),f(x))
\end{align*}
to extend it to a regular mapping $\psi:C\rightarrow\R^n$ satisfying $||\psi(x)||\leq 2\varepsilon$ for $x\in C$. Finally consider
\begin{equation*}
    \widetilde{f}:=\sigma(h,\psi).
\end{equation*}
It is a well defined mapping agreeing with $f$ on $Z\cap C$. Moreover, as $h$ approaches $f$ in the compact-open topology, $\varepsilon$ approaches zero, so $\widetilde{f}$ does approach $f$ in the compact-open topology.
\end{proof}
\end{prop}

\section{$\CC^\infty$ approximation}

We now head towards the proof of Theorem \ref{smooth_app_thm}.
\begin{lem}\label{C^inf_extension_lem}
Let $X$ be a non-singular affine variety, $D_0\subset X$ an open set, $Z\subset X$ a Zariski closed subvariety of $X$ which is faithful at every point of $Z\cap D_0$ (see Remark \ref{faith} for the definition of faithfulness) and $f:D_0\rightarrow \R^n$ a smooth bounded mapping such that $f\vert_{Z\cap D_0}$ is constantly equal to $0$. Then $f$ can be approximated in the $\CC^\infty$ topology by restrictions $\widetilde{f}\vert_{D_0}$ of regular mappings $\widetilde{f}:X\rightarrow\R^n$, which satisfy
\begin{enumerate}
    \item $\widetilde{f}$ vanishes on $Z$,
    \item $\sup_{x\in X}||\widetilde{f}(x)||\leq 2\sup_{x\in D_0}||f(x)||$.
\end{enumerate}
\begin{proof}
First of all, without loss of generality we might substitute for $X$ its non-singular algebraic compactification (which exists by the virtue of resolution of singularities) to assume that $X$ is compact.

To avoid the trivial case we assume that $f$ is not constanly equal to zero. Fix a relatively compact open subset $D_1\subset D_0$ and another relatively compact open subset $D_2\subset D_1$. According to \cite[Proposition 4.1]{Banecki+Kucharz}, $f$ can be approximated in the $\mathcal{C}^\infty$ topology by restrictions of regular mappings on $X$ which vanish on $Z$. Let $h:X\rightarrow \R^n$ be such a mapping, we choose it to be so close to $f$ that
\begin{equation*}
    \sup_{x\in D_1} ||h(x)||<2\sup_{x\in D_0}||f(x)||.
\end{equation*}

It is easy to construct, using for example Urysohn's lemma and Whitney approximation theorem, a smooth function $g:X\rightarrow [0,1]$, such that
\begin{align*}
    g(x)&=1  \text{ for }x\in D_2, \\
    g(x)&=0  \text{ for }x\in X\backslash D_1.
\end{align*}
Let $\widetilde{g}:X\rightarrow \R$ be a close regular approximation of $g$ in the $\mathcal{C}^\infty$ topology, which exists thanks to a version of the Stone-Weierstrass theorem. Finally define $\widetilde{f}:=h\widetilde{g}$. If $\widetilde{g}$ is sufficiently close to $g$ then indeed
\begin{equation*}
    \sup_{x\in X}||\widetilde{f}(x)||\leq 2\sup_{x\in D_0}||f(x)||.
\end{equation*}
On the other hand $\widetilde{f}\vert_{D_2}$ approximates $f\vert_{D_2}$ in the $\mathcal{C}^\infty$ topology, so as $D_2$ is an arbitrary relatively compact open subset of $D_0$ the proof is finished.
\end{proof}
\end{lem}

\begin{proof}[Proof of Theorem \ref{smooth_app_thm}]
We prove the stronger version of Theorem \ref{smooth_app_thm} from Remark \ref{faith}. Let $X,Y,C,Z,D,f:C\rightarrow Y$ be as in the theorem. We know from Theorem \ref{urr_equivalences} that $Y$ is strongly malleable, so let $(n,M,N,\sigma,\tau)$ be a strong dominating spray over $Y$. Thanks to Lemma \ref{ball_in_M}, we assume $Y\times \mathbb{B}(0,1)\subset M$. Choose a relatively compact open subset $D_0\subset D$. Applying Proposition \ref{cont_app_prop}, we can find a close regular approximation $h$ of $f$ in the compact-open topology which agrees with $f$ on $Z\cap C$. If the approximation is close enough, then $(h,f)(D_0)\subset N$, so it makes sense to define
\begin{align*}
    \varphi&:D_0\rightarrow \R^n \\
    \varphi(x)&:=\tau(h(x),f(x)).
\end{align*}
Now define
\begin{equation*}
    \varepsilon:=\sup_{x\in D_0} ||\varphi(x)||.
\end{equation*}
Again assuming that $h$ is sufficiently close to $f$ we infer $\varepsilon<1/2$. We can now apply Lemma \ref{C^inf_extension_lem} to $\varphi$ to find a regular mapping $\psi:X\rightarrow\R^n$, vanishing on $Z$, such that $\psi\vert_{D_0}$ closely approximates $\varphi$ in the $\CC^\infty$ topology and satisfying
\begin{equation*}
    \sup_{x\in X} ||\psi(x)||<2\varepsilon.
\end{equation*}
Finally consider
\begin{align*}
    \widetilde{f}&:C\rightarrow Y, \\
    \widetilde{f}&:=\sigma(h,\psi).
\end{align*}
It is a well defined mapping agreeing with $f$ on $Z\cap C$. Moreover, its restriction to $D_0$ is a close approximation of $f\vert_{D_0}$ in the $\CC^\infty$ topology, and simultaneously it is close to $f$ on $C$ in the compact-open topology. As $D_0$ was an arbitrary relatively compact open subset of $D$, the proof is finished.
\end{proof}

\section{Counterexamples}\label{sec:examples}
Let us begin by constructing the example mentioned in Section \ref{sec:intro}, of a malleable variety which is not retract rational. The construction is based on the following two results:

\begin{obs}\label{rr_connected}
Let $Y\subset \R^n$ be a compact non-singular retract rational real affine variety. Then it is connected.
\end{obs}

\begin{prop}\label{cubic_is_malleable}
Let $n\geq 3$ and let $P\in\R[y_1,\dots,y_{n+1}]$ be a cubic homogeneous irreducible polynomial such that the complex hypersurface $Y^\mathbb{C}\subset \mathbb{C}\mathbb{P}^n$ given by $P(y)=0$ is non-singular. Assume that the real variety $Y=Y^\mathbb{C}\cap \R\mathbb{P}^n$ is non-empty. Then it is malleable.
\end{prop}
\begin{rem}
This a real algebraic version of the corresponding result \cite[Theorem 1.1]{counterex_new} in the complex case.
\end{rem}

Given the two results, it is not difficult to come up with an example of an irreducible malleable variety which is not retract rational:
\begin{ex}
The surface in $\R\mathbb{P}^3$ given by $w(x^2+y^2)-z(z-w)(z-2w)=0$ is malleable, yet it is not retract rational.
\begin{proof}
One easily checks that it is disconnected, yet it satisfies the assumptions of Proposition \ref{cubic_is_malleable}.
\end{proof}
\end{ex}

\begin{proof}[Proof of Observation \ref{rr_connected}]
Let $V\subset Y$ be a Zariski open and dense subset such that there exists a Zariski open set $W\subset \R^m$ and two regular mappings $i:V\rightarrow W$ and $r:W\rightarrow V$ satisfying $r\circ i=\id_V$. The mapping $r$, considered as a mapping into $\R^n$ is rational and bounded, so it admits an extension $\widetilde{r}:\widetilde{W}\rightarrow Y$ to a Zariski open set $\widetilde{W}\subset \R^m$ satisfying $\codim(\R^m\backslash\widetilde{W})>1$. The set $\widetilde{W}$ is connected, so the image of $\widetilde{r}$ is contained in a connected component of $Y$. Then $V$ is also contained in a connected component of $Y$. If there existed another connected component of $Y$, it would not be Zariski dense in $Y$, so it would have to consist of non-central points. This contradicts non-singularity of $Y$.
\end{proof}

\begin{proof}[Proof of Proposition \ref{cubic_is_malleable}]
Define
\begin{multline*}
    U:=\{(y,z)\in Y\times Y:\\
    y\neq z\text{ and the line connecting $y$ and $z$ is neither tangent to $y$ nor $z$}\}.
\end{multline*}
The set $U$ is Zariski open. We claim that for every $y\in Y$, there exists some $z\in Y$ such that $(y,z)\in U$. Suppose this is not true for some $y\in Y$ and consider the projection 
\begin{equation*}
    \pi:\R\mathbb{P}^n\backslash\{y\}\rightarrow \R\mathbb{P}^{n-1}.
\end{equation*}
From the assumption it follows that for $z\in Y$ we either have that $z$ lies in the tangent space $T_y Y$ to $Y$ at $y$ or it is a singular point of $\pi\vert_{Y\backslash \{y\}}$. As the image of $T_y Y\backslash\{y\}$ through $\pi$ is a hyperplane, it follows that $\pi(Y\backslash\{y\})$ is of semialgebraic dimension at most $n-2$. Hence a generic fibre of $\pi\vert_{Y\backslash\{y\}}$ is one-dimensional. This implies that $Y$ is contained in $T_y Y$, so $Y$ is a hyperplane. Now as $P$ is of degree three it is reducible which is a contradiction.

Following \cite{Kollár_2002}, consider the regular mapping $\phi:U\rightarrow Y$ which to a pair of points of $Y$ associates the third intersection of the line connecting them with $Y$. Note that for $(y,z)\in U$ it satisfies $\phi(y,z)=\phi(z,y)$, and the mapping $w\mapsto \phi(w,\phi(y,z))$ maps $y$ to $z$ and it is an isomorphism between a neighbourhood of $y$ and a neighbourhood of $z$, with its inverse given by itself.
As $Y$ is irreducible, this in particular shows that it is quasi-homogeneous in the sense that $\RR(Y)_y$ is isomorphic to $\RR(Y)_z$ for every two points $y,z\in Y$.

Thanks to \cite[Theorem 1]{Kollár_2002}, $Y$ is unirational. By \cite[Proposition 1.5]{dominating} this is equivalent to an existence of a regular mapping $\psi:\R^m\rightarrow Y$ submersive at $0$. Let $y:=\psi(0)$ and let $V:=\{z\in Y:(z,y)\in U\}$. Define
\begin{align*}
    \sigma:V\times \R^m\dashrightarrow V \\
    \sigma(z,v):=\phi\left(\psi(v),\phi(y,z)\right)
\end{align*}
As noticed before, the mapping $w\mapsto \phi(w,\phi(y,z))$ is an isomorphism between a neighbourhood of $y$ and a neighbourhood of $z$, so $v\mapsto \sigma(z,v)$ is defined on a neighbourhood of $0$, it takes $0$ to $z$ and it is a submersion at $0$. Reasoning as in the proof of \cite[Proposition 6.2]{Banecki+Kucharz} we can now find a regular mapping $\varepsilon:V\rightarrow \R$ such that the triple $(V\times \R^n,p,s)$ is a dominating spray over $V$, where $p$ is the projection onto the first coordinate and $s$ is defined by
\begin{equation*}
    s(z,v):=\sigma\left(z,\varepsilon(z)\frac{v}{1+||v||^2}\right).
\end{equation*}
This shows that $V$ is a malleable variety. As $Y$ is quasi-homogeneous, it follows that each point of $Y$ admits a Zariski open neighbourhood which is malleable. It now follows from \cite[Theorem 6.3]{Banecki+Kucharz} that $Y$ is malleable.
\end{proof}

We now proceed to provide a few counterexamples showing ways in which Theorems \ref{cont_app_thm} and \ref{smooth_app_thm} cannot be improved further.

As shown in the following two examples, the assumptions about $C$ being closed and semialgebraic in Theorem \ref{smooth_app_thm} are essential. However, it is not clear whether the assumption that $Z\cap C$ is compact can be dropped.

\begin{ex}
In this example all assumptions of Theorem \ref{cont_app_thm} are satisfied except for closedness of $C$. Take $X:=\R^2,Y:=\R\backslash \{0\}, Z:=\R\times \{0\}\subset X$ and 
\begin{equation*}
    C:=\{(x,y)\in \R^2: x<0, y<0\}\cup \{(x,y)\in \R^2:1\leq x\leq 2, y\leq 0\}.
\end{equation*}
If we define $f:C\rightarrow Y$ as
\begin{equation*}
    f(x,y)=
    \begin{cases}
        1 & x\leq 0 \\
        x & x>0
    \end{cases}
\end{equation*}
then $f$ is homotopically trivial, and $f\vert_{Z\cap C}$ is equal to the projection onto the first variable. On the other hand, if $\widetilde{f}:C\rightarrow Y$ is a regular approximation of $f$ equal to $f$ on $Z\cap C$, then by definition it extends to a Zariski open neighbourhood $U$ of $C$. This is impossible, as on the one hand $\widetilde{f}$ maps $C$ into the component $x>0$ of $Y$ and yet it maps every point of $Z\cap U\cap \partial C\neq \emptyset$ into the other component.
\end{ex}

\begin{ex}
In this example all assumptions of Theorem \ref{cont_app_thm} are satisfied except for semialgebraicity of $C$. Take $X:=\R^2\backslash\{(0,0)\},Y:=\R\backslash \{0\}, Z:=\R\times \{0\}\backslash\{(0,0)\}\subset X$ and 
\begin{equation*}
    C:=\{(x,y)\in \R^2\backslash\{(0,0)\}: x\leq 0, y\leq -e^{\frac{1}{x}}\}\cup \{(x,y)\in \R^2:1\leq x\leq 2, y\leq 0\}.
\end{equation*}
Note that $C$ is constructed in such a way that every open semialgebraic set containing a segment of the form $\{(x,y)\in\R^2: -\varepsilon<x<0,y=0\}$ intersects $C$ non-vacuously.

If we again define $f:C\rightarrow Y$ as
\begin{equation*}
    f(x,y)=
    \begin{cases}
        1 & x\leq 0 \\
        x & x>0
    \end{cases}
\end{equation*}
then $f$ is homotopically trivial, and $f\vert_{Z\cap C}$ is equal to the projection onto the first variable. Once again, if $\widetilde{f}:C\rightarrow Y$ is a regular approximation of $f$ equal to $f$ on $Z\cap C$, then by definition it extends to a Zariski open neighbourhood $U$ of $C$. This is again impossible, as the set $\widetilde{f}^{-1}(\{x<0\})$ is semialgebraic and contains a small segment of the form $\{-\varepsilon<x<0,y=0\}$, and on the other hand it has empty intersection with $C$ as the set $C$ gets mapped into the component $\{x>0\}$ of $Y$.
\end{ex}

It is worth noting that in Theorems \ref{smooth_app_thm} and \ref{cont_app_thm}, unlike in the case of malleable varieties (see the generalised approximation-interpolation property in \cite[Definition 1.5]{Banecki+Kucharz} for comparison), we cannot expect approximation by mappings extending to the entire variety $X$ even if $g$ does extend this way as a regular mapping:

\begin{ex}
Take $X=\R^2$, $Z$ and $Y$ both equal to the unit circle $\s^1$ centered at the origin of $\R^2$ and $C$ equal to the closed ball of radius $\frac{1}{2}$ centered at $(1,0)$. Take $f$ to be the radial projection onto the circle, restricted to $C$. The mapping is homotopic to a regular one as in fact it is null homotopic. It is also equal to the identity when restricted to $C\cap \s^1$. Nonetheless it cannot be approximated by mappings regular on $\R^2$ equal to the identity on $\s^1$ since such mappings do not exist for topological reasons.
\end{ex}

Lastly, in Theorem \ref{smooth_app_thm}, even if $C$ is a smooth manifold with boundary and $f:C\rightarrow Y$ is a smooth mapping, we cannot expect the regular mappings to approximate $f$ in the $\CC^\infty$ topology on the entire set $C$. In the following example the approximation cannot be even of class $\CC^1$ at a particular boundary point of $C$:
\begin{ex}
Take $X:=\R^2$, $Y:=\R$, $Z:=\R\times\{0\}$ and
\begin{equation*}
    C:=\overline{\B}((0,1),1)\cup \overline{\B}((3,0),1).
\end{equation*}
Define $f$ as 
\begin{equation*}
    f(x,y):=
        \begin{cases}
            x & x\leq 1 \\
            0 & x>1
        \end{cases}
\end{equation*}
Now, given any regular function $\widetilde{f}:C\rightarrow \R$, vanishing on $Z\cap C$, its partial derivative $\pdv{\widetilde{f}}{x}$ must also vanish on $Z\cap C$, in particular it must vanish at the origin. Hence it cannot approximate $f$ in the $\CC^\infty$ topology, as $\pdv{f}{x}(0,0)\neq 0$.
\end{ex}

\section{Acknowledgements}
The author was partially supported by the National Science Centre (Poland) under grant number 2022/47/B/ST1/00211. 

\printbibliography

\end{document}